\documentclass[11pt, english]{article}

\usepackage[margin= 1.9 cm]{geometry}
\usepackage[T1]{fontenc}
\usepackage{amsthm}
\usepackage{amsmath}
\usepackage{amssymb}
\usepackage{setspace}
\usepackage{mathtools}
\usepackage{verbatim}
\usepackage{booktabs}
\usepackage{graphicx}
\usepackage[hidelinks]{hyperref}
\usepackage{cleveref}
\usepackage{graphicx}
\usepackage{appendix}
\usepackage[inline]{enumitem}
\usepackage{framed}
\usepackage{subcaption}
\usepackage{microtype,xurl}
\usepackage{caption}
\usepackage{listings}
\usepackage{xcolor}

\usepackage{floatrow}
\usepackage[T1]{fontenc}

\lstdefinestyle{verifier}{
  language=Python,
  basicstyle=\ttfamily\fontsize{8.1}{10.2}\selectfont,
  keywordstyle=\bfseries,
  commentstyle=\itshape,
  stringstyle=\ttfamily,
  showstringspaces=false,
  keepspaces=true,
  columns=fullflexible,
  breaklines=true,
  breakatwhitespace=false,
  numbers=left,
  numberstyle=\tiny,
  numbersep=7pt,
  xleftmargin=1.6em,
  aboveskip=0.8em,
  belowskip=0.8em,
  tabsize=4,
  upquote=true
}

\theoremstyle{plain}

\newtheorem*{thm*}{Theorem}
\newtheorem{thm}{Theorem}
\Crefname{thm}{Theorem}{Theorems}

\newtheorem*{lem*}{Lemma}
\newtheorem{lem}[thm]{Lemma}
\Crefname{lem}{Lemma}{Lemmas}

\newtheorem*{claim*}{Claim}

\Crefname{claim}{Claim}{Claims}

\newtheorem{prop}[thm]{Proposition}
\Crefname{prop}{Proposition}{Propositions}

\Crefname{cor}{Corollary}{Corollaries}

\Crefname{conj}{Conjecture}{Conjectures}

\Crefname{qn}{Question}{Questions}

\newtheorem{obs}[thm]{Observation}
\Crefname{obs}{Observation}{Observations}

\newtheorem{ex}[thm]{Example}
\Crefname{ex}{Example}{Examples}

\theoremstyle{definition}

\Crefname{prob}{Problem}{Problems}

\Crefname{defn}{Definition}{Definitions}

\newtheorem*{defn*}{Definition}

\theoremstyle{remark}

\makeatletter
\newenvironment{claimproof}[1][Proof]{\par
	\pushQED{\qed}%
	
	\normalfont \topsep6\p@\@plus6\p@\relax
	\trivlist
	\item[\hskip\labelsep
	\textit{#1}\@addpunct{.}~]\ignorespaces
}{%
	\popQED\endtrivlist\@endpefalse
}
\makeatother

\newcommand{\ang}[1]{
    \langle #1 \rangle
}

\expandafter\def\expandafter\normalsize\expandafter{%
    \normalsize
    \setlength\abovedisplayskip{8pt}
    \setlength\belowdisplayskip{8pt}
    \setlength\abovedisplayshortskip{4pt}
    \setlength\belowdisplayshortskip{4pt}
}

\usepackage[square,sort,comma,numbers]{natbib}
\providecommand{\doi}[1]{\href{https://doi.org/#1}{doi:\nolinkurl{#1}}}

 \setlist[itemize]{leftmargin=*}

 \newcommand{\EE}{\mathbb E}

 \newcommand{\file}[1]{\mbox{\nolinkurl{#1}}}

\newcommand{\Span}{\operatorname{span}}

\newif\ifexternalbibliography
\IfFileExists{sunflower_references.bib}{\externalbibliographytrue}{\externalbibliographyfalse}

\title{Proof of the Erd\H{o}s-Szemer\'edi sunflower conjecture}
\author{
Matija Buci\'c\thanks{Faculty of Mathematics, University of Vienna, Vienna, Austria. This research was funded in whole or in part by the Austrian Science Fund (FWF) 10.55776/AST2416125. Email: \href{mailto:matija.bucic@univie.ac.at} { \ttfamily matija.bucic@univie.ac.at}. } 
\and
Fankang He%
        \thanks{Faculty of Mathematics, University of Vienna, Vienna, Austria. Email: \href{mailto:fankang.he@univie.ac.at} {\ttfamily fankang.he@univie.ac.at}.}
\and
Hyunwoo Lee%
        \thanks{Department of Mathematical Sciences, KAIST, South Korea and Extremal Combinatorics and Probability Group (ECOPRO), Institute for Basic Science (IBS).
        E-mail: {\ttfamily hyunwoo.lee@kaist.ac.kr.} Supported by the National Research Foundation of Korea (NRF) grant funded by the Korea government(MSIT) No. RS-2023-00210430, and the Institute for Basic Science (IBS-R029-C4).}
}

\date{}

\begin{document}
\maketitle
\begin{abstract}
 In the 1970s, Erd\H{o}s and Szemer\'edi conjectured that for any $n \ge k-1\ge 2$, any collection of subsets of $[n]$ not containing a $k$-sunflower has size at most $c_k^n$ for some $c_k<2$. In the $k=3$ case, this follows from the work of Croot, Lev, and Pach and Ellenberg and Gijswijt on the cap-set problem, through a reduction due to Alon, Shpilka, and Umans. We prove the conjecture for every $k \ge 3$.
\end{abstract}

\section{Introduction}
A collection of distinct sets is said to be a \emph{sunflower} if all their pairwise intersections are the same set. Sunflowers provide a simple way to separate the common structure of several sets from the parts that do not interact. This makes them useful both for finding order in large families and for simplifying collections of combinatorial witnesses. They are part of a basic toolkit in extremal combinatorics and Ramsey theory, with countless applications. They are also a tool for proving lower bounds for monotone circuits deciding whether a graph has a clique of specified size or a perfect matching and for studying thresholds of monotone properties~\cite{alonshpilkaumans2013,rao2026story}. In additive combinatorics, they are connected to the cap-set problem, while their vector and multicolored versions underpin combinatorial approaches to fast matrix multiplication \cite{alonshpilkaumans2013,blasiaketal2017}. They also underlie the compression of the hitting set instances in parameterized algorithms \cite{bannachtantau2018}, while information theory supplies a useful tool for arguing about sunflowers through Shannon's noiseless coding theorem; see~\cite{rao2020}.

Many of the above connections are underpinned by natural extremal questions concerning sunflower-free families. The most famous of these are arguably the Erd\H{o}s-Rado sunflower conjecture \cite{erdosrado1960} and the Erd\H{o}s-Szemer\'edi sunflower conjecture \cite{erdos1970extremal,erdos1971topics,erdosszem1978}. The latter is concerned with the very natural question of how large can the family $\mathcal F$ of subsets of $[n]=\{1,2,\ldots,n\}$ be if it does not contain a sunflower consisting of $k$ sets? 

This question has been repeated many times over the years, with the oldest reference the authors could find dating back to 1970, both by Erd\H{o}s with various coauthors \cite{erdos1970extremal,erdos1971analysis,erdos1971topics,erdos1971unsolved,erdos1975finite,erdos1976graphs,erdos1994setsystems,erdos1996favourite,erdosszem1978}, and by other authors \cite{deuberetal1997,chung1997problems,
chunggraham1998,kostochka2000systems,alonshpilkaumans2013,
blasiaketal2017,naslundsawin2017,
hegedus2018improved,hegedus2018,mubayiwang2018,
frankltokushige2018,alweissetal2021,alonholzman2023,
naslund2023,franklpachpalvolgyi2024,crootlevpach2025survey,
tangzhang2025,petrturek2026,luoyangzhu2026,BradacBucicSudakov2023}.
Erd\H{o}s and Szemer\'edi~\cite{erdosszem1978}, proved that for any fixed $k \ge 3$ the answer is at most $2^{n-\Omega(\sqrt n)}$ and conjectured that the answer should be exponentially smaller than $2^n$. 
Using improved uniform sunflower bounds of Kostochka \cite{kostochka1996}, Deuber, Erd\H{o}s, Gunderson, Kostochka, and Meyer~\cite{deuberetal1997} strengthened this to $2^{n-\sqrt{n\log\log n}/{\log\log\log n}}$ for large $n$.
The breakthrough on the uniform problem due to Alweiss, Lovett, Wu, and Zhang~\cite{alweissetal2021}, from 2019, improved the bound to $2^{n-n/(\log n)^{1+o(1)}}$. Subsequent refinements \cite{rao2020,bellchueluechawarnke2021} (in particular the one due to Rao \cite{rao2020}) give a slightly sharper bound of $2^{n-\Omega(n/\log n)}$, see \cite[Theorem 4.1]{alweissetal2021}.

For the special case of $k=3$, an important connection to additive combinatorics was established by Alon, Shpilka, and Umans~\cite{alonshpilkaumans2013}, who showed that an exponential improvement over the trivial bound for cap sets in $\mathbb F_3^n$ would imply the Erd\H{o}s-Szemer\'edi sunflower conjecture in the $k=3$ case.  Here, the polynomial-method breakthroughs of Croot, Lev, and Pach~\cite{crootlevpach2017} and Ellenberg and Gijswijt~\cite{ellenberggijswijt2017} in 2016 supplied this cap-set estimate. Naslund and Sawin~\cite{naslundsawin2017} then gave a direct argument giving the base of the exponent of ${3}/{2^{2/3}} \approx 1.89$ in the upper bound. Their proof, using Tao's slice-rank formulation of the polynomial method~\cite{tao2016}, turns the absence of a sunflower into a diagonal tensor. The slice rank of that tensor measures the size of the family, while an expansion of its coordinatewise definition gives an upper bound. The same circle of ideas connects cap sets, tensors, and the group-theoretic approach to matrix multiplication \cite{blasiaketal2017}.

We prove the Erd\H{o}s-Szemer\'edi sunflower conjecture for every $k \ge 3$.

\begin{thm}\label{thm:set}
For every integer $k\ge3$ and every integer $n\ge1$, any subset of $2^{[n]}$ not containing a sunflower consisting of $k$ sets has size at most 
\begin{equation*}
 3\left(2-\frac{1}{16k^3}\right)^n.
\end{equation*}
\end{thm}

We note that for $n<k-1$ the assumption is vacuous, so one may take the whole set $2^{[n]}$. For $n \ge k-1$, the above theorem, combined with the trivial bound of $2^n-1$ (coming from observing that singletons and the empty set always make a sunflower), shows that one can obtain a universal upper bound of $c_k^{n}$ for some $c_k<2$.

The best known lower bound is due to Deuber, Erd\H{o}s, Gunderson, Kostochka, and Meyer~\cite{deuberetal1997}, who constructed $k$-sunflower-free families of size at least $\bigl(2-\frac{O(\log\log k)}k\bigr)^n$ for arbitrarily large $n$, showing that \Cref{thm:set} is not far from best possible for large $k$. 

\section{Proof ideas}\label{sec:proof-ideas}
We prove Theorem~\ref{thm:set} by first bounding sunflower-free families of vectors and then transferring this bound to set families. For $k=3$, Alon, Shpilka, and Umans~\cite[Section 2]{alonshpilkaumans2013} showed that the Erd\H{o}s--Szemer\'edi sunflower conjecture is equivalent to a power-saving bound in the alphabet size for the vector problem. Naslund and Sawin~\cite[Theorem~3]{naslundsawin2017} subsequently gave an explicit quantitative reduction from cap-set bounds, corresponding to the vector problem over a three-letter alphabet. We follow this approach for arbitrary $k$, with a modification of the reduction that improves the quantitative bound in the parameter range used here.

To see how vectors arise, suppose that a ground set is split into $n$ disjoint blocks of size $D$. A set choosing exactly one element from each block can be recorded as a vector in $[D]^n$. Such sets form a sunflower precisely when, in every block, their chosen elements are either all equal or pairwise distinct. This gives the following formulation, which we will refer to as the vector sunflower problem of Alon, Shpilka, and Umans.

\begin{defn*}
    For integers $k\ge3$ and $D,n\ge1$, a \emph{$k$-sunflower} in $[D]^n$ is a collection of $k$ pairwise distinct vectors $\mathbf{x}_0,\ldots,\mathbf{x}_{k-1}\in[D]^n$ such that, in each coordinate, their entries are either all equal or pairwise distinct. Let $V_k(D,n)$ be the maximum size of a subset of $[D]^n$ containing no $k$-sunflower. 
\end{defn*}

For the reduction to the set problem, it suffices to find one fixed alphabet size $D$ and a constant $1\le C<D$ such that $V_k(D,n)\le C^n$ for every $n\ge1$. In Section~\ref{sec:transfer}, we prove this implication for arbitrary $k$ using a probabilistic refinement of the block construction underlying the reduction of Naslund and Sawin. We average over random block choices within each family of sets of a fixed size, and then sum the resulting bounds over the possible sizes. This gives the quantitative saving needed for Theorem~\ref{thm:set}.

For all integer $D \geq 3$, Naslund and Sawin~\cite[Theorem~2]{naslundsawin2017} proved 
$$
    V_3(D,n) \le \left(\frac{3(D-1)^{2/3}}{2^{2/3}}\right)^n,
$$
which establishes the power-saving upper bound for the case $k = 3$ of the vector sunflower problem.

Their proof uses a Fourier character expansion of a symmetric function that is nonzero precisely when its three entries are all equal or pairwise distinct. Each term in this expansion has a constant factor in at least one of the three variables. After multiplying across coordinates, some vector therefore contributes at most $2n/3$ nontrivial characters, and counting these character products gives the exponent $2/3$ of the alphabet size in their bound. 

Our main new ingredient is an interpolation construction of a local function for arbitrary $k$. We choose its nonzero values on all-distinct tuples to depend on the entries, and use successive interpolation to arrange that every term has a polynomial factor of degree at most $k-2$ in one of the last $k-1$ variables, with the remaining factor independent of that variable. These low-degree factors all lie in a common space of dimension $k-1$. After multiplying across coordinates, a pigeonhole argument and a count of the resulting polynomial factors give the exponent $1-1/(k-1)$ below.

\begin{samepage}
\begin{thm}\label{thm:vector}
    Let $k\ge3$ and $D, n\ge1$ be integers. Then
    $$
        V_k(D,n)
        \le
        \bigl(C_k\cdot D^{1-1/(k-1)}\bigr)^n,
    $$
    where $C_k=(2k)^{2/(k-1)}$.
\end{thm}
\end{samepage}

The saving in the power of $D$ is crucial in the reduction step. In particular, taking $D=16k^2$ gives $C_kD^{1-1/(k-1)}<D$, and the reduction in Section~\ref{sec:transfer} then yields Theorem~\ref{thm:set}. We now describe the main ideas behind the vector bound.

We aim to bound the size of a sunflower-free family in $[D]^n$ by constructing a function of $k$ vectors from $[D]^n$ with two properties. On the family, it should equal $1$ when all input vectors coincide and vanish otherwise. It should also admit an expansion in which every term has a factor from a small common polynomial space in one of the last $k-1$ vectors, multiplied by a function independent of that vector. An elementary linear-algebra argument will then bound the size of the family by the dimension of this space.

To obtain the first property, we exploit the fact that the sunflower condition can be checked separately in each coordinate. We define a local function that equals $1$ when all entries are equal, is nonzero when they are pairwise distinct, and vanishes otherwise. Multiplying this function across coordinates gives the required values on a sunflower-free family. The challenge is to obtain the second property as well. Here we use the freedom to choose the nonzero values on all-distinct tuples. Successive interpolation will construct these values so that the resulting function has an expansion with the required low-degree factors.

To do this, we start with the product of the indicators that each entry equals the first, and replace its factors one at a time by interpolation polynomials. The interpolation makes the remaining error vanish as soon as an entry repeats an earlier one. At the end, the error is therefore nonzero precisely on all-distinct tuples. Subtracting it gives the desired local function, together with an expansion in which each term has a factor of degree at most $k-2$ in one of the last $k-1$ variables.

We then multiply across coordinates and expand into monomials. The pigeonhole principle ensures that each monomial has low exponents in at least $\lceil n/(k-1)\rceil$ coordinates of some vector. Grouping terms according to such a vector gives the required decomposition with a common polynomial space. Counting the monomials spanning this space yields Theorem~\ref{thm:vector}.

\section{Vector sunflower-free sets}\label{sec:vector}

The goal of this section is to prove Theorem~\ref{thm:vector} by using the polynomial method, namely a dimension-counting argument. Let $U$ be an arbitrary finite set. For a $\mathbb{Q}$-vector space $\mathcal{P}$ of functions from $U$ to $\mathbb{Q}$, we denote its dimension by $\dim\mathcal{P}$. Also, to avoid technical confusion, if we write a function $f: \ang{X}^r \to \mathbb{Q}$, it means $f$ is an $r$-variable $\mathbb{Q}$-valued function whose inputs are $r$-tuples of elements of $X$. Throughout this section, we fix positive integers $k \geq 3$ and $D, n \geq 1$. We also call a tuple \emph{monochromatic} if its entries are all the same, and \emph{rainbow} if its entries are pairwise distinct.

Let $U$ be a finite set and let $V \subseteq U$. We say that a function $f:\ang{U}^k\to\mathbb{Q}$ is \emph{$k$-diagonal on $V$} if for all $x_0,\ldots,x_{k-1}\in V$, $f(x_0,\ldots,x_{k-1}) \neq 0$ if and only if $(x_0, \ldots, x_{k-1})$ is monochromatic. The following lemma bounds $|V|$ by the dimension of a common space of functions on $U$, provided that $f$ admits a decomposition into terms of a particular form. Each term must factor as a function from this space evaluated at one of the variables $x_1,\ldots,x_{k-1}$, multiplied by a function independent of that variable. In our application, this common space will consist of polynomial functions.

For readers familiar with slice rank, the following lemma refines the usual (see \cite[Lemma~1]{tao2016}) argument with diagonal tensors by exploiting a common space for the single-variable factors. This motivates seeking a decomposition that separates a function of one variable in each term, while keeping all such functions in a small common space.

\begin{lem}\label{lem:dim-argument}
    Let $k \ge 2$, and let $V \subseteq U$ be finite sets. Let $f:\ang{U}^k \to \mathbb{Q}$ be $k$-diagonal on $V$. Suppose that there is a finite-dimensional $\mathbb{Q}$-vector space $\mathcal{P}$ of functions from $U$ to $\mathbb{Q}$, with a basis $p_1,\ldots,p_d$, such that
    $$
        f(x_0,\ldots,x_{k-1})
        =
        \sum_{i=1}^{k-1}\sum_{j=1}^{d}
        p_j(x_i)
        G_{i,j}(x_0,\ldots,\widehat{x_i},\ldots,x_{k-1})
    $$
    for some functions $G_{i,j}:\ang{U}^{k-1}\to\mathbb{Q}$, where the hat denotes an omitted variable. Then
    $$
        |V| \le \dim \mathcal{P}.
    $$
\end{lem}

\begin{proof}
    Suppose for a contradiction that $|V|>d$. The homogeneous system
    $$
        \bigg\{ \sum_{x\in V}h(x)p_j(x)=0 : j \in [d] \bigg\}
    $$
    has $|V|$ unknowns and $d$ equations, so it admits a nonzero solution $h:V\to\mathbb{Q}$. Choose $a\in V$ with $h(a)\ne0$ and consider
    $$
        S
        =
        \sum_{x_1,\ldots,x_{k-1}\in V}
        f(a,x_1,\ldots,x_{k-1})
        \prod_{i=1}^{k-1}h(x_i).
    $$
    Substitute the assumed decomposition of $f$. In each summand indexed by $(i,j)$, the function $G_{i,j}$ is independent of $x_i$, so summing first over $x_i$ gives the factor
    $$
        \sum_{x_i\in V}h(x_i)p_j(x_i)=0,
    $$
    hence $S=0$. On the other hand, since $f$ is $k$-diagonal on $V$, the only nonzero contribution to $S$ comes from $x_1=\cdots=x_{k-1}=a$. Thus $S=f(a, \ldots, a)h(a)^{k-1}\neq 0$, a contradiction.
\end{proof}

In our application of Lemma~\ref{lem:dim-argument}, we take $U=[D]^n$ and let $V$ be a $k$-sunflower-free subset of $[D]^n$.
To reflect the coordinatewise structure of sunflowers, we say a function $T:\ang{[D]^n}^k\to\mathbb{Q}$ is a \emph{sunflower tester} if for all $\mathbf{x}_0, \dots, \mathbf{x}_{k-1} \in [D]^n$, the value $T(\mathbf{x}_0, \ldots, \mathbf{x}_{k-1}) \neq 0$ if and only if for all $i\in[n]$, the tuple $((\mathbf{x}_0)_i, \ldots, (\mathbf{x}_{k-1})_i)$ is either monochromatic or rainbow.

We think of such $T$ as a test function for $k$-sunflowers: it is nonvanishing if and only if the input vectors are all the same or they form a $k$-sunflower. The vector definition of a sunflower and Lemma~\ref{lem:dim-argument} give the following consequence.

\begin{lem}\label{lem:sunflower-dim}
    A sunflower tester $T:\ang{[D]^n}^k\to\mathbb{Q}$ is $k$-diagonal on every $k$-sunflower-free subset $V\subseteq[D]^n$. In particular, if $T$ admits the decomposition in Lemma~\ref{lem:dim-argument} with a finite-dimensional $\mathbb{Q}$-vector space $\mathcal{P}$ of functions from $[D]^n$ to $\mathbb{Q}$, then $|V|\le\dim\mathcal{P}$ for every such $V$. Consequently,
    $$
        V_k(D,n)\le\dim\mathcal{P}.
    $$
\end{lem}

\begin{proof}
    Let $V\subseteq[D]^n$ be $k$-sunflower-free, and let $\mathbf{x}_0,\ldots,\mathbf{x}_{k-1}\in V$ satisfy $T(\mathbf{x}_0,\ldots,\mathbf{x}_{k-1})\ne0$. By our definition of a sunflower tester, their entries at each coordinate are either all equal or pairwise distinct. If the vectors were pairwise distinct, they would therefore form a $k$-sunflower, contrary to the assumption on $V$. Hence two of the vectors coincide. Their entries agree at every coordinate, ruling out the pairwise distinct alternative at every coordinate. Thus $\mathbf{x}_0=\cdots=\mathbf{x}_{k-1}$. Since a sunflower tester is nonzero on every monochromatic tuple, this shows that $T$ is $k$-diagonal on $V$. If $T$ admits the stated decomposition, Lemma~\ref{lem:dim-argument} gives $|V|\le\dim\mathcal{P}$. Maximizing over all $k$-sunflower-free subsets $V\subseteq[D]^n$ yields the bound on $V_k(D,n)$.
\end{proof}

\subsection{Extract low degree polynomials via interpolation} \label{subsec:interpolation}

Our goal is to construct a sunflower tester $T:\ang{[D]^n}^k\to\mathbb{Q}$ that admits the decomposition in Lemma~\ref{lem:dim-argument} with a low-dimensional space $\mathcal{P}$, and then apply Lemma~\ref{lem:sunflower-dim}. Since the sunflower condition is determined coordinatewise, a natural approach is to choose a function $q:[D]^k\to\mathbb{Q}$ encoding the allowed equality patterns in a single coordinate and define $T: \ang{[D]^n}^k \to \mathbb{Q}$ as a product of $q$ values at the $k$-tuple consisting of the $i$-th coordinates of the input. 

For a function $c:\ang{[D]}^k\to \mathbb{Q}$, which is nonvanishing on rainbow tuples, we define a function $q_c:\ang{[D]}^k\to\mathbb{Q}$ as
$$
    q_c(a_0, a_1,\ldots,a_{k-1})
    =
    \begin{cases}
        1,
        & \text{if } (a_0,a_1,\ldots,a_{k-1})\text{ is monochromatic},\\
        c(a_0, a_1, \dots, a_{k-1}),
        & \text{if } (a_0,a_1,\ldots,a_{k-1})\text{ is rainbow},\\
        0,
        & \text{otherwise},
    \end{cases}
$$
to reflect the structure of sunflowers. We also define $T_c: \ang{[D]^n}^k \to \mathbb{Q}$ as
$$
    T_c(\mathbf{x}_0, \dots, \mathbf{x}_{k-1}) = \prod_{i\in [n]} q_c((\mathbf{x}_0)_i, \dots, (\mathbf{x}_{k-1})_i).
$$
We will specify the function $c$ that fits our purpose later. For any such choice, it is immediate that the function $T_c$ is a sunflower tester.

\begin{obs}\label{obs:T_c-proper}
    Let $c:\ang{[D]}^k\to \mathbb{Q}$ be nonvanishing on rainbow tuples. Then $T_c:\ang{[D]^n}^k\to\mathbb{Q}$ is a sunflower tester.
\end{obs}

\begin{proof}
    For every $\mathbf{x}\in[D]^n$, each factor in $T_c(\mathbf{x},\ldots,\mathbf{x})$ equals $1$, so $T_c(\mathbf{x},\ldots,\mathbf{x})=1$. If $\mathbf{x}_0,\ldots,\mathbf{x}_{k-1}$ have entries at some coordinate that are neither all equal nor pairwise distinct, the corresponding factor in the defining product for $T_c$ is zero. Hence $T_c(\mathbf{x}_0,\ldots,\mathbf{x}_{k-1})=0$. Thus $T_c$ satisfies the condition for being a sunflower tester.
\end{proof}

The main task is to choose $c$ so that $T_c$ admits the decomposition in Lemma~\ref{lem:dim-argument} with a low-dimensional space $\mathcal{P}$. To this end, we will find a polynomial representation of $q_c$ in which every monomial has an exponent at most $k-2$ in at least one of the variables $a_1,\ldots,a_{k-1}$. After taking products across coordinates, a pigeonhole argument will turn this restriction into the required decomposition of $T_c$.

Since $q_c$ is defined on the finite discrete domain $\ang{[D]}^k$, it admits many polynomial representations. In other words, distinct polynomials represent the same function whenever they agree at every point of this domain. We will use successive interpolation to choose the values of $q_c$ on all-distinct tuples and obtain a polynomial representation with the desired one variable with low exponent in each monomial property.

\begin{lem}\label{lem:q_c-interpolation}
    There exists a function $c:\ang{[D]}^k\to\mathbb{Q}$ that is nonvanishing on rainbow tuples such that $q_c$ admits the representation
    \begin{equation}\label{eq:low-degree-decomp}
        q_c(a_0,\ldots,a_{k-1})
        =
        \sum_{i=1}^{k-1}\sum_{j=0}^{i-1}
        a_i^j F_{i,j}(a_0,\ldots,\widehat{a_i},\ldots,a_{k-1})
    \end{equation}
    for all $a_0,\ldots,a_{k-1}\in[D]$, where $F_{i,j}:\ang{[D]}^{k-1}\to\mathbb{Q}$ are functions and the hat denotes an omitted variable.
\end{lem}

\begin{proof}
    Throughout the proof, we use $a_0$ as the reference value. For $a_0,t\in[D]$, let
    $$
        e_{a_0}(t)=\mathbf{1}_{t=a_0}.
    $$
    For $h\in[k]$, write
    $$
        \Delta_h
        =
        \prod_{i=h}^{k-1}e_{a_0}(a_i),
    $$
    where an empty product equals $1$, so $\Delta_k=1$. In particular, $\Delta_1$ is the indicator of monochromatic tuples, and for $h\in[k-1]$,
    $$
        \Delta_h=e_{a_0}(a_h)\Delta_{h+1}.
    $$

    The function $\Delta_1$ is a natural starting point since it already equals $1$ on monochromatic tuples and vanishes whenever the entries are neither monochromatic nor rainbow. We will modify its values on rainbow tuples to obtain $q_c$, which admits the desired low-degree decomposition.
    For fixed $a_0$, any polynomial representing $e_{a_0}(t)$ on $[D]$ has $D-1$ distinct roots and takes the value $1$ at $a_0$, so its degree is at least $D-1$. Rather than representing each factor in full, we use successive interpolation to extract terms of low degree while keeping track of rainbow prefixes in the remainder. 
    At step $h$, we extract a term of degree at most $h-1$ in $a_h$, leaving a remainder of the form $R_h(a_0,\ldots,a_h)\Delta_{h+1}$, where $R_h$ is nonzero precisely when $(a_0,\ldots,a_h)$ is rainbow. At the final step, the remainder is therefore nonzero precisely on rainbow tuples, so subtracting it from $\Delta_1$ gives the desired function $q_c$ with the required decomposition. The following claim makes this construction precise.

    \begin{claim*}
        There exist functions $F_{i,j}:\ang{[D]}^{k-1}\to\mathbb{Q}$ for $1\le i\le k-1$ and $0\le j\le i-1$, and functions $R_h:\ang{[D]}^{h+1}\to\mathbb{Q}$ for $0\le h\le k-1$, such that
        $$
            \Delta_1
            =
            \sum_{i=1}^{h}\sum_{j=0}^{i-1}
            a_i^j F_{i,j}(a_0,\ldots,\widehat{a_i},\ldots,a_{k-1})
            +
            R_h(a_0,\ldots,a_h)\Delta_{h+1}
        $$
        for every $h\in\{0,\ldots,k-1\}$ and all $a_0,\ldots,a_{k-1}\in[D]$. Moreover, $R_h(a_0,\ldots,a_h)$ is nonzero precisely when $(a_0,\ldots,a_h)$ is rainbow.
    \end{claim*}

    \begin{claimproof}
        We construct these functions by induction on $h$. For $h=0$, the sum is empty and we take $R_0(a_0)=1$. The identity and the required property of $R_0$ both hold, since every tuple with a single entry is rainbow.

        Now let $h\in[k-1]$ and suppose that the functions have been constructed through step $h-1$. The part that remains to be decomposed is
        $$
            R_{h-1}(a_0,\ldots,a_{h-1})\Delta_h
            =
            R_{h-1}(a_0,\ldots,a_{h-1})
            e_{a_0}(a_h)\Delta_{h+1}.
        $$
        To extract a term with the desired degree restriction, we replace        $e_{a_0}(a_h)$ by a polynomial of degree at most $h-1$ in $a_h$. Multiplying this polynomial by $R_{h-1}(a_0,\ldots,a_{h-1})\Delta_{h+1}$ gives the desired new term of low degree in $a_h$, while multiplying the replacement error (the indicator minus the polynomial) by the same factor supplies the new remainder. To keep the new remainder zero on prefixes with repeated entries, we want this error to vanish whenever $a_h$ repeats one of $a_0,\ldots,a_{h-1}$. Since $R_{h-1}$ already vanishes unless the prefix $(a_0,\ldots,a_{h-1})$ is rainbow, we only need to interpolate at the $h$ distinct entries of a rainbow prefix.

        Accordingly, for each fixed prefix $(a_0,\ldots,a_{h-1})\in[D]^h$, define the polynomial $P_h(a_0,\ldots,a_{h-1};t)\in\mathbb{Q}[t]$ by
        $$
            P_h(a_0,\ldots,a_{h-1};t)
            =
            \begin{cases}
                \displaystyle\prod_{i=1}^{h-1}\frac{t-a_i}{a_0-a_i},
                & \text{if }(a_0,\ldots,a_{h-1})\text{ is rainbow},\\[3pt]
                0,
                & \text{otherwise}.
            \end{cases}
        $$
        The case distinction depends only on the fixed prefix, so $P_h$ is a polynomial of degree at most $h-1$ in $t$, whose coefficients are functions of the prefix. On a rainbow prefix, it takes the value $1$ at $a_0$, and the value $0$ at $a_1,\ldots,a_{h-1}$, and hence agrees with $e_{a_0}$ at every earlier entry.
        We split the next factor as
        $$
            e_{a_0}(a_h)
            =
            P_h(a_0,\ldots,a_{h-1};a_h)
            +
            \bigl(e_{a_0}(a_h)-P_h(a_0,\ldots,a_{h-1};a_h)\bigr).
        $$
        The contribution from the first term has the required form. Indeed, both $R_{h-1}$ and $\Delta_{h+1}$ are independent of $a_h$, so expanding $P_h$ gives
        $$
            \begin{aligned}
                &R_{h-1}(a_0,\ldots,a_{h-1})
                P_h(a_0,\ldots,a_{h-1};a_h)\Delta_{h+1}=
                \sum_{j=0}^{h-1}
                a_h^j F_{h,j}(a_0,\ldots,\widehat{a_h},\ldots,a_{k-1})
            \end{aligned}
        $$
        for some functions $F_{h,j}:\ang{[D]}^{k-1}\to\mathbb{Q}$. The contribution from the second term is $R_h(a_0,\ldots,a_h)\Delta_{h+1}$, where
        $$
            R_h(a_0,\ldots,a_h)
            =
            R_{h-1}(a_0,\ldots,a_{h-1})
            \bigl(e_{a_0}(a_h)-P_h(a_0,\ldots,a_{h-1};a_h)\bigr).
        $$
        Substituting these two contributions into the inductive representation gives the required identity at step $h$, without changing any previously chosen $F_{i,j}$.

        It remains to check the required property of $R_h$. If $(a_0,\ldots,a_{h-1})$ is not rainbow, then $R_{h-1}=0$ by the induction hypothesis. If this prefix is rainbow and $a_h$ repeats an earlier entry, the interpolation conditions make the difference in parentheses zero. If $a_h$ is distinct from every entry of this rainbow prefix, then $e_{a_0}(a_h)=0$, while every factor in $P_h(a_0,\ldots,a_{h-1};a_h)$ is nonzero. Hence $R_h(a_0,\ldots,a_h)$ is nonzero precisely when $(a_0,\ldots,a_h)$ is rainbow. This completes the induction and proves the claim.
    \end{claimproof}

    Applying the claim with $h=k-1$ and using $\Delta_k=1$, we obtain
    $$
        \Delta_1
        =
        \sum_{i=1}^{k-1}\sum_{j=0}^{i-1}
        a_i^j F_{i,j}(a_0,\ldots,\widehat{a_i},\ldots,a_{k-1})
        +
        R_{k-1}(a_0,\ldots,a_{k-1}).
    $$
    We now define
    $$
        c(a_0,\ldots,a_{k-1})
        =
        -R_{k-1}(a_0,\ldots,a_{k-1}).
    $$
    By the claim, $c$ is nonvanishing on rainbow tuples and vanishes elsewhere. Since $\Delta_1$ is the indicator of all entries being equal, the definition of $q_c$ gives
    $$
        q_c(a_0,\ldots,a_{k-1})
        =
        \Delta_1-R_{k-1}(a_0,\ldots,a_{k-1}).
    $$
    The required representation follows. This completes the proof of \Cref{lem:q_c-interpolation}.
\end{proof}

We remark that when bounding the size of a vector sunflower-free set, we will use Lagrange interpolation to represent each coefficient function $F_{i,j}$ in \eqref{eq:low-degree-decomp} by a polynomial of degree at most $D-1$ in each of its variables. These representatives remain independent of $a_i$, so the degree restriction in $a_i$ is preserved.

\subsection{Dimension counting}\label{subsec:dim-count}

Fix a function $c$ satisfying the conclusion of Lemma~\ref{lem:q_c-interpolation}. Then $q_c$ admits the representation in \eqref{eq:low-degree-decomp}, and $T_c$ is a sunflower tester by Observation~\ref{obs:T_c-proper}. Expanding the coordinatewise product defining $T_c$ yields the required decomposition with a space $\mathcal{P}$ of polynomial functions spanned by monomials having sufficiently many low-degree coordinates. Thus, by Lemma~\ref{lem:sunflower-dim}, bounding $V_k(D,n)$ reduces to estimating $\dim\mathcal{P}$. Counting the monomials spanning this space gives the following lemma.

\begin{lem}\label{lem:degree-counting}
    If $D\ge k-1$, then
    $$
        V_k(D,n)
        \le
        \sum_{s=0}^{(1-1/(k-1))n}
        \binom{n}{s}(D-k+1)^s(k-1)^{n-s}.
    $$
\end{lem}

\begin{proof}
    Choose $c$ and the functions $F_{i,j}$ as in Lemma~\ref{lem:q_c-interpolation}. To apply Lemma~\ref{lem:dim-argument}, we need to separate the dependence on an entire vector from the dependence on the remaining vectors. The representation \eqref{eq:low-degree-decomp} gives a low-degree factor in one coordinate, but coefficient functions arising at other coordinates may still depend on other entries of the same vector. We use Lagrange interpolation to expand these functions into monomials, whose dependence on each vector can be separated.

    Specifically, by Lagrange interpolation in each variable, every $F_{i,j}$ admits a polynomial representation of degree at most $D-1$ in each of its variables. This representation remains independent of $a_i$, so multiplying it by $a_i^j$ preserves the exponent $j$ in that variable. Since $j\le i-1\le k-2\le D-1$, substituting these representatives into \eqref{eq:low-degree-decomp} gives a polynomial representation of $q_c$ in which every individual exponent is at most $D-1$, and every monomial has an exponent at most $k-2$ in at least one of $a_1,\ldots,a_{k-1}$.

    Substitute this representation into the coordinatewise product defining $T_c$ and expand. Every resulting monomial has the form
    $$
        M
        =
        \prod_{i=0}^{k-1}\prod_{r=1}^{n}
        (\mathbf{x}_i)_r^{d_{i,r}},
    $$
    where $0\le d_{i,r}\le D-1$. For each coordinate $r\in[n]$, there is some $i\in[k-1]$ such that $d_{i,r}\le k-2$. By the pigeonhole principle, some $i\in[k-1]$ satisfies this inequality for at least $ n/(k-1)$ coordinates. Consequently, the number of coordinates with $d_{i,r}\ge k-1$ is at most $n-n/(k-1)=(1-1/(k-1))n.$

    Let $\mathcal{P}$ be the $\mathbb{Q}$-vector space of functions on $[D]^n$ spanned by the monomial functions
    $$
        \mathbf{x}\longmapsto\prod_{r=1}^{n}x_r^{d_r},
        \qquad \mathbf{x}=(x_1,\ldots,x_n),
    $$
    whose integer exponents satisfy $0 \leq d_r \leq D-1$ for every $r\in [n]$, and
    $$
        \bigl|\{r\in[n]:d_r\ge k-1\}\bigr|
        \le (1-1/(k-1))n.
    $$
    For each monomial $M$, choose an index $i\in[k-1]$ supplied by the pigeonhole argument and write
    $$
        M
        =
        \left(\prod_{r=1}^{n}(\mathbf{x}_i)_r^{d_{i,r}}\right)
        \left(
            \prod_{\substack{0\le \ell\le k-1\\ \ell\ne i}}
            \prod_{r=1}^{n}(\mathbf{x}_\ell)_r^{d_{\ell,r}}
        \right).
    $$
    The first factor is an element of $\mathcal{P}$ evaluated at $\mathbf{x}_i$, while the second factor is independent of $\mathbf{x}_i$. Grouping these terms and expressing their first factors in a basis of $\mathcal{P}$ therefore gives the decomposition required in Lemma~\ref{lem:dim-argument}. Since $T_c$ is a sunflower tester by Observation~\ref{obs:T_c-proper}, Lemma~\ref{lem:sunflower-dim} yields $V_k(D,n)\le\dim\mathcal{P}$.

    Finally, we count the monomials spanning $\mathcal{P}$ according to the number $s$ of coordinates whose exponents are at least $k-1$. There are $\binom{n}{s}$ choices for these coordinates, $D-k+1$ choices for each of their exponents, and $k-1$ choices for each remaining exponent. Therefore,
    $$
        V_k(D,n)
        \le \dim\mathcal{P}
        \le
        \sum_{s=0}^{(1-1/(k-1))n}
        \binom{n}{s}(D-k+1)^s(k-1)^{n-s}.
    $$
    This completes the proof.
\end{proof}

Combined with a simple application of Chernoff's inequality, Lemma~\ref{lem:degree-counting} yields the proof of Theorem~\ref{thm:vector}.

\begin{proof}[Proof of Theorem~\ref{thm:vector}]
     If $D\le 4k^2$, the desired bound follows from the trivial bound $V_k(D,n)\le D^n$, so we may assume $D>4k^2$. We estimate the sum in Lemma~\ref{lem:degree-counting} by a multiplicative Chernoff bound \cite[Theorem~A.1.12]{alonspencer2016} telling us that $\mathbb{P}(\textrm{Bin}(n,(k-1)/D) \ge n/(k-1))\le (4k^2/D)^{n/(k-1)}$. In particular,  
            \begin{align*} 
            V_k(D,n) &\le D^n \sum_{s=0}^{(1-1/(k-1))n} \binom{n}{s}\left(1-\frac{k-1}D\right)^s\left(\frac{k-1}D\right)^{n-s}\le \bigl((4k^2)^{1/(k-1)} \cdot D^{1-1/(k-1)}\bigr)^{n},
            \end{align*}    
    as desired.
\end{proof}

\section{Reduction from vectors to sets}
\label{sec:transfer}

We now transfer the vector sunflower bound to the set family setting. A probabilistic refinement of the argument of Naslund and Sawin~\cite{naslundsawin2017} yields the following quantitative version. 

\begin{prop}\label{prop:transfer}
    Let $D \geq 2$ be an integer, and suppose that for some constant $C$ with $1 \leq C <D$, we have $V_k(D,n) \leq C^n,$ for every $n \geq 1$. Let $\rho = C/D$. Then, for every integer $N \geq 1$, any subset of $2^{[N]}$ not containing a $k$-sunflower has size at most 
    $$
    (1+\rho^{-1})(1+\rho^{1/(D-1)})^N. 
    $$
\end{prop}
\begin{proof}
    Take $\mathcal{F} \subseteq 2^{[N]}$ with no $k$-sunflower. For $0 \leq s \leq N$, let $\mathcal{F}_s$ be the sets in $\mathcal{F}$ of size $s$, i.e. $\mathcal{F}_s=\{F \in \mathcal{F}: |F| =s\}$. Now fix $s$, and let $t := \min \{s, \lfloor \frac{N-s}{D-1} \rfloor\}$. Take a permutation of $[N]$ uniformly at random. Let its first $Dt$ elements form ordered sets $B_1,\ldots,B_t$, each of size $D$, and its next $s-t$ elements form a set $R$. The existence of such $B_1,\ldots,B_t, R$ is guaranteed by our choice of $t$. If $t > 0$, then we define $\mathcal{G}$ by the family of sets obtained by taking the union of $R$ with exactly one element from each $B_i$, i.e., $\mathcal{G}=\{R \cup \{b_1,b_2,\ldots,b_t\}: b_i \in B_i\}$. If $t=0$, set $\mathcal{G}=\{R\}$. In either case, $|\mathcal{G}|=D^t$. Then we can view each element in $\mathcal{G}$ as a vector in $[D]^t$ by identifyng each element $v = R\cup \{b_{1, j_1}, \dots, b_{t, j_t}\} \in \mathcal{G}$ with the tuple $(j_1, \dots, j_t)\in [D]^t$, where $b_{i, j_i}$ is the $j_i$-th element of $B_i$ in the prescribed order for each $i\in [t]$. 
    
    Since $\mathcal{F}_s$ contains no $k$-sunflower, the set of vectors in $[D]^t$ corresponding to $\mathcal{F}_s\cap\mathcal{G}$ contains no $k$-sunflower either. Indeed, if $k$ such vectors formed a $k$-sunflower, then in each block $B_i$, their chosen elements would be either all the same or pairwise distinct. 
    In the former case, that common element would belong to the core, while in the latter case, $B_i$ would contribute nothing to any pairwise intersection. 
    Since every set also contains the common fixed part $R$, the corresponding $k$ sets would therefore form a $k$-sunflower in $\mathcal{F}_s$, a contradiction.
    Hence, $|\mathcal{F}_s \cap \mathcal{G}| \leq V_k(D,t) \leq C^t$. 
    
    By symmetry, every $s$-element subset of $[N]$ has the same probability of belonging to $\mathcal{G}$. 
    Indeed, for any two $s$-element subsets $S_1, S_2\subseteq [N]$, there exists a permutation of $[N]$ sending $S_1$ to $S_2$. 
    Since $\mathcal{G}$ is generated from a uniformly random permutation of $[N]$, relabeling the ground set by any fixed permutation leaves the distribution of $\mathcal{G}$ unchanged. 
    This gives $\mathbb{P}(S_1 \in \mathcal{G}) = \mathbb{P}(S_2 \in \mathcal{G})$. 
    Hence, for every $s$-element subset $S \subseteq [N]$, we have $\mathbb{P}(S \in \mathcal{G}) = |\mathcal{G}|/\binom{N}{s} = D^t/\binom{N}{s}$. 
    
    By linearity of expectation, we have 
    $$
    \EE|\mathcal{F}_s \cap \mathcal{G}| = \sum_{F \in \mathcal{F}_s} \mathbb{P}(F \in \mathcal{G}) = |\mathcal{F}_s| \frac{D^t}{\binom{N}{s}}. 
    $$
    Combining this with $|\mathcal{F}_s \cap \mathcal{G}| \leq C^t$, we get 
    $$
    |\mathcal{F}_s| \leq \binom{N}{s} \frac{C^t}{D^t} = \binom{N}{s} \rho^t \leq \binom{N}{s} \left(\rho^s + \rho^{-1+\frac{N-s}{D-1}} \right). 
    $$
    Therefore, 
    \begin{align*}
    |\mathcal{F}| = \sum_{s=0}^N |\mathcal{F}_s| &\leq \sum_{s=0}^N \binom{N}{s} \left(\rho^s + \rho^{-1+\frac{N-s}{D-1}} \right) \\
    &= (1+\rho)^N + \rho^{-1}(1+\rho^{1/(D-1)})^N \\
    &\leq (1+\rho^{-1})(1+\rho^{1/(D-1)})^N, 
    \end{align*}
    where the second inequality holds by the binomial theorem. 
\end{proof}

Applying the proposition with a suitable choice of $D$ finishes the proof of \Cref{thm:set}.

\begin{proof}[Proof of \Cref{thm:set}]
Take $D=16k^2$ and set
$C=C_kD^{1-1/(k-1)}$. By \Cref{thm:vector},
we have $V_k(D,n)\le C^n$. Writing
$$
 \rho=\frac{C}{D}
 =\left(\frac{4k^2}{D}\right)^{1/(k-1)}
 =4^{-1/(k-1)},
$$
we have $1/2\le\rho<1$, and hence
$1+\rho^{-1}\le3$. Moreover,
$$
 \log\rho=-\frac{\log4}{k-1}\le-\frac1k.
$$
Consequently,
$$
 \rho^{1/(D-1)}
 \le \exp\left(-\frac{1}{k(D-1)}\right)
 \le 1-\frac{1}{1+k(D-1)}
 \le 1-\frac{1}{16k^3},
$$
where the second inequality uses
$e^{-1/x}\le1-1/(1+x)$ for $x>0$.
Applying \Cref{prop:transfer} gives
$$
 (1+\rho^{-1})(1+\rho^{1/(D-1)})^n
 \le 3\left(2-\frac{1}{16k^3}\right)^n,
$$
as desired.
\end{proof}

\section{Concluding remarks}\label{sec:conclusion}

In this paper, we settle the Erd\H{o}s-Szemer\'edi sunflower conjecture by showing that any $k$-sunflower-free family of subsets of $[n]$ is exponentially smaller than the total number of subsets $2^n$. 

The closest sibling of the Erd\H{o}s-Szemer\'edi sunflower conjecture is the Erd\H{o}s-Rado sunflower conjecture \cite{erdosrado1960}, which trades the dependency on the size of the ground set for the dependency on the maximum size of the sets. More precisely, they conjectured that for any fixed $k$, any collection of sets of size at most $r$ that does not contain a $k$-sunflower has size at most $c_k^r$ for some constant $c_k>1$. The current record towards this problem is $O(k \log r )^r$, coming from a breakthrough paper by Alweiss, Lovett, Wu and Zhang \cite{alweissetal2021}, with subsequent refinements by Rao \cite{rao2020} (Tao \cite{tao} found a different proof of this result using Shannon entropy) and Bell, Chueluecha and Warnke \cite{bellchueluechawarnke2021}. It was mentioned, already in 1971, in \cite{erdos1971topics}, that the Erd\H{o}s--Rado sunflower conjecture implies the Erd\H{o}s-Szemer\'edi one. The implication in the other direction remains open, although Alon, Shpilka, and Umans \cite{alonshpilkaumans2013} showed that if one can show an exponential upper bound in the vector variant of the problem (namely $V_k(D,n)$) which is independent of the alphabet size $D$, then one would prove the Erd\H{o}s-Rado Conjecture. Our bounds, while yielding a polynomial saving in alphabet size, remain heavily alphabet-dependent.

We made no particular efforts to optimize the base of the exponent in our bounds and have opted instead to focus on making the argument as approachable as possible. For example, we are using a coarse version of the multiplicative Chernoff bound in our proof of \Cref{thm:vector}, and are very wasteful in our estimates around that argument. There are several further optimizations readily available, such as, for example, noticing that in the proof of \Cref{thm:vector} we put every low-degree factor in the common space $\Span\{1,t,\ldots,t^{k-2}\}$, whereas the factor selected at stage $j$ belongs to $U_j=\Span\{1,t,\ldots,t^{j-1}\}$, which has dimension $j$. With a bit more work, one can exploit this observation to improve the base of the exponent in the upper bound on the vector sunflower problem slightly. We note that this still falls short of improving the best-known base in the $k=3$ case in the set family problem, due to Naslund and Sawin \cite{naslundsawin2017}.
We note that in the vector sunflower problem on $[D]^n$, for $k=3$ our argument does improve the $O(D^{2/3})^n$ bound from \cite{naslundsawin2017} to $O(\sqrt{D})^n$. This also translates to a similar minor improvement in the restricted-union variant of the uniform Erd\H{o}s-Rado setting via the reduction of Heged\H{u}s~\cite{hegedus2018}.

\subsection{Applications, connections and limitations}
The new bound on the vector sunflower problem of Alon, Shpilka and Umans~\cite{alonshpilkaumans2013} (\Cref{thm:vector}) extends to indexed multicolored sunflowers~\cite{alonshpilkaumans2013,naslundsawin2017} and limits the hard instances supplied by the vector-based reduction for arithmetic property testing in~\cite[Theorem 3.3]{havivxie2017}. The set system bound (\Cref{thm:set}) gives consequences for equal subset sums~\cite{alonshpilkaumans2013} and harmonic least-common-multiple patterns~\cite{tangzhang2025,luoyangzhu2026}, which we explain in a bit more detail below. Further consequences concern biased measures and weighted families~\cite{luoyangzhu2026}, and the density of sunflower-free families of subspaces~\cite{ihringerkupavskii2026,otal2026}. The polynomial argument also connects to cap sets and tricolored sum-free sets~\cite{ellenberggijswijt2017,blasiaketal2017}, while the set bound gives alternative exponential estimates for weak $\Delta$-systems, equidistant sets and odd sunflowers with an odd number of members~\cite{naslund2023,franklpachpalvolgyi2024}. The polynomial detector functions constructed in our proof also give subrank bounds with interpretations in matrix multiplication and exact entanglement conversion~\cite{christandlvranazuiddam2021,vranachristandl2015}.
Stronger conclusions would be needed for several further applications. In particular, one would need robust hitting estimates for lifting, circuit complexity, Kneser-graph packing and the Alon--Jaeger--Tarsi problem~\cite{derezendevinyals2025,alweissetal2021}, bounds independent of the ambient number of coordinates for random linear codes~\cite{silas2026}, and control of the core size for restricted-intersection and shadow-tomography problems~\cite{janzeretal2026}.

\noindent \textbf{Harmonic least-common-multiple patterns.}
The following problem was posed by Erd\H{o}s~\cite{erdos1970extremal} (recorded as Problem~856 in Bloom's collection~\cite{bloom856}) in the same 1970 paper where the first mention of the Erd\H{o}s-Szemer\'edi sunflower problem appears. Let $h_k(N)$ be the largest harmonic sum $\sum_{a\in A}1/a$ over sets $A\subseteq[N]$ containing no $k$ distinct integers whose pairwise least common multiples are all equal. Tang and Zhang establish a connection between this problem and sunflower-free set families~\cite{tangzhang2025}. Writing $\delta_k=1/(16k^3)$, \Cref{thm:set} gives, for every fixed $k\ge3$,
$$
    h_k(N)\ll_k(\log N)^{1-\delta_k}.
$$
Indeed, the squarefree-multiplier double count of Luo, Yang, and Zhu~\cite{luoyangzhu2026} gives
$$
 \left(\sum_{\substack{q\le N\\q\text{ squarefree}}}\frac{1}{q}\right)
 \sum_{a\in A}\frac{1}{a}
 \le\sum_{u\le N^2}\frac{F_k(\omega(u))}{u}
 \ll_k(\log N)^{2-\delta_k},
$$
where $F_k(r)$ denotes the maximum size of a $k$-sunflower-free family in $2^{[r]}$, and $\omega(u)$ counts the distinct prime divisors of $u$. The parenthesized sum is of order $\log N$, proving the claim. Thus we obtain a fixed logarithmic-power saving for every $k\ge4$. For $k=3$, the Naslund--Sawin \cite{naslundsawin2017} bound already gives an asymptotic exponent $3/2^{2/3}-1=0.889881\ldots$~\cite{tangzhang2025}.

There is also a Ramsey consequence. If $P$ is the product of the first $X$ primes, encode $S\subseteq[X]$ by $P/\prod_{i\in S}p_i$. A sunflower then gives integers with equal pairwise least common multiples. The monochromatic-sunflower estimate below and $\log P=O(X\log X)$ show that every $c$-coloring of $[N]$ contains such a monochromatic $k$-tuple once
$$
    N\ge\exp(A_k\log c\,\log\log c),
$$
for a sufficiently large constant $A_k$ depending only on $k$ and all sufficiently large $c$.

\noindent \textbf{Equal subset sums.}
By \Cref{thm:set}, every $c$-coloring of $2^{[N]}$ contains a monochromatic $k$-sunflower if
$$
 N>32k^3\log(3c).
$$
Indeed, a largest color class then has more than $3(2-\delta_k)^N$ members. Coloring subsets of a sequence in a finite abelian group $G$ by their sums and cardinalities, and removing the common core, gives $k$ pairwise disjoint nonempty index sets of equal size and equal sum. Thus $A_k\log(2|G|)$ terms suffice, for a constant $A_k$ depending only on $k$. The cardinality color ensures that all petals have the same positive size. For three petals, this recovers the known consequence used by Alon, Shpilka and Umans to rule out the Coppersmith--Winograd conjecture about disjoint equivoluminous subsets~\cite{alonshpilkaumans2013}.

Similarly, a subset of $[M]$ of size at least $A_k\log(2M)$ contains $k$ pairwise disjoint nonempty subsets with a common positive sum $s$. To see this, color by sum and cardinality, using at most $(NM+1)(N+1)$ colors when the original set has size $N$. Taking unions gives $0,s,\ldots,ks$ among its subset sums. Both deductions apply for every fixed $k$.

\noindent \textbf{Induced Erd\H{o}s--Sauer problem.}
As a final application, we mention a topic which has only recently been connected to sunflower problems. The famous Erd\H{o}s--Sauer problem \cite{erdos1975progress} asks, for fixed $k\ge3$, for the maximum number of edges in an $n$-vertex graph not containing a $k$-regular subgraph. In a remarkable recent paper, Janzer and Sudakov~\cite{janzersudakov2023} solved this problem by showing that the answer is $\Theta_k(n\log\log n)$, matching the lower bound of Pyber, R\"odl and Szemer\'edi~\cite{pyberrodlszemeredi1995} from 1995. 
In a subsequent work, Chakraborti, Janzer, Methuku, and Montgomery \cite{ChakrabortiJanzerMethukuMontgomery2026} observed and used a connection between the Erd\H{o}s-Sauer problem and the Erd\H{o}s-Rado sunflower conjecture and obtained even essentially tight dependencies also on the regularity parameter $k$.
None of these arguments, however, provide an induced $k$-regular subgraph. The induced variant of the problem, an arguably even more natural question, was already posed by Szemer\'edi in the 1970s; see Erd\H{o}s~\cite[Section~3]{erdos1975progress}. Recently, we discovered a reduction showing that \Cref{thm:vector} implies that $\Theta_k(n\log\log n)$ is the answer in the induced variant as well. Since the reduction is highly non-trivial and of a very different flavor from the arguments in this paper, we will present it in an upcoming short note~\cite{induced-erdos-sauer}.

A recent preprint by Karingula and Lovett~\cite{Karingula-Lovett} proves that, for $k \ge 4$, every tensor faithfully realizing the set-sunflower relation on the full Boolean cube has maximal slice rank, over every field and with arbitrary nonzero values on the prescribed support. They also obtain comparable lower bounds when the relation is prescribed only on pairwise distinct inputs. They discuss extensive use of AI, including multi-agent setups, with particular emphasis on expressing their sunflower argument directly in terms of slice decompositions. Their obstruction concerns the set-sunflower relation. Our construction instead gives a tester for vector sunflowers, and we transfer the resulting vector bound to set families; thus their result does not directly apply to our tester.

\medskip 

\subsection*{Statement on the use of AI}
The proof of \Cref{thm:vector} was found by ChatGPT 6 Astra Pro, following a short discussion, which started by asking it to generalize the $k=3$ case argument in \cite{naslundsawin2017}, which we needed for our original target of proving the induced variant of the Erd\H{o}s--Sauer problem. The original argument it produced used the slice-rank terminology due to Tao (as used in \cite{naslundsawin2017}); we rewrote it in the polynomial method language as we believe this would make the argument slightly more approachable to a reader unfamiliar with the slice-rank method and since no deep result or property of slice rank was used. The authors have verified all the results in the paper and take full responsibility for the content. The proof has been verified in Lean 4.35; the verification will be made available on the Palomar repository.

\subsection*{Acknowledgements}
This work was performed during a visit of Hyunwoo Lee to the University of Vienna. Hyunwoo Lee is grateful to the University of Vienna for its support and hospitality.

\bibliographystyle{amsplain_initials_nobysame}
\bibliography{sunflower-references}

\end{document}